\documentclass[a4paper,11pt]{article}
\usepackage[english]{babel}
\usepackage{amsmath}
\usepackage{amsfonts}
\usepackage{amssymb}
\usepackage{amsthm}
\usepackage{epsfig,graphics}
\usepackage{graphicx}
\usepackage{hyperref}

\title{A Tutte polynomial for toric arrangements}

\author{Luca Moci}

\begin{document}

\newtheorem{te}{Theorem}[section]
\newtheorem{lem}[te]{Lemma}
\newtheorem{pr}[te]{Proposition}
\newtheorem{prob}[te]{Problem}
\newtheorem{co}[te]{Corollary}
\newtheorem{cj}[te]{Conjecture}
\theoremstyle{definition}
\newtheorem{re}[te]{Remark}
\newtheorem{ex}[te]{Example}

\newcommand{\faktor}{\frac}
\newcommand{\PT}{\mathcal{C}_{0}(\Phi)}
\newcommand{\Pt}{\mathcal{C}_{0}(\Theta)}
\newcommand{\Tquo}{{\mathfrak{h}}/{\langle\Phi^\vee\rangle}}
\newcommand{\sym}{\mathfrak{S}}


\newcommand{\lc}{\Lambda_C}
\newcommand{\lcb}{\overline{\Lambda_C}}
\newcommand{\cs}{\mathcal{S}}
\newcommand{\s}{\mathcal{S}}
\newcommand{\vs}{\mathcal{V_S}}
\newcommand{\vt}{\mathcal{V_T}}
\newcommand{\vsp}{\vs^p}
\newcommand{\vtq}{\vt^q}
\newcommand{\rx}{\mathcal{R}_X}
\newcommand{\pcn}{\mathbb{P}(\mathbf{N}_T(C))}
\newcommand{\pc}{\mathbb{P}_C}
\newcommand{\zx}{\mathbf{Z}_X}
\newcommand{\yx}{\mathbf{Y}_X}
\newcommand{\zt}{z^{\mathcal{T}}}
\newcommand{\dn}{\mathbf{D}_\mathcal{N}}
\maketitle

\begin{flushright}
    \emph{Ad Alessandro Pucci, \\che ha ripreso in mano il timone della propria vita.}
\end{flushright}

\begin{abstract}
We introduce a multiplicity Tutte polynomial $M(x,y)$, with applications to zonotopes and toric arrangements. We prove that
$M(x,y)$ satisfies a deletion-restriction recursion and has
positive coefficients. The characteristic polynomial and the
Poincar\'{e} polynomial of a toric arrangement are shown to be
specializations of the associated polynomial $M(x,y)$, likewise
the corresponding polynomials for a hyperplane arrangement are
specializations of the ordinary Tutte polynomial. Furthermore,
$M(1,y)$ is the Hilbert series of the related
discrete Dahmen-Micchelli space, while $M(x,1)$ computes the volume
and the number of integer points of the associated zonotope.
\end{abstract}

\section{Introduction}
The Tutte polynomial is an invariant naturally associated to a
matroid and encoding many of its features, such as the number
of bases and their \emph{internal and external activity}
(\cite{Tu}, \cite{Cr}, \cite{li}). If the matroid is realized by a
finite list of vectors, it is natural to consider the
arrangement obtained by taking the hyperplane orthogonal to
each vector. One associates to the poset of the intersections of the
hyperplanes its \emph{characteristic
polynomial}, which provides a rich combinatorial and
topological description of the arrangement (\cite{OSb},
\cite{Za}). This polynomial can be obtained as a specialization
of the Tutte polynomial.

Given a complex torus $T=(\mathbb{C}^*)^n$ and a finite list $X$ of
characters, i.e. elements of $Hom (T,\mathbb{C}^*)$, we
consider the arrangement of hypersurfaces in $T$ obtained by
taking the kernel of each element of $X$. To understand the
geometry of this \emph{toric arrangement} one needs to
describe  the poset $\mathcal{C}(X)$ of the \emph{layers}, i.e.
connected components of the intersections of the
hypersurfaces (\cite{DPt}, \cite{ERS}, \cite{Mo}, \cite{Mw}). Clearly this
poset depends also on the arithmetics of $X$, and not only on
its linear algebra: for example, the kernel of the identity
character $\lambda$ of $\mathbb{C}^*$  is the point $t=1$, but
the kernel of $2\lambda$ has equation $t^2=1$, hence is made of
two points. Therefore we have no chance to get the
characteristic polynomial of $\mathcal{C}(X)$ as a
specialization of the ordinary Tutte polynomial $T_X(x,y)$  of
$X$. In this paper we
define a polynomial $M_X(x,y)$ that specializes to the
characteristic polynomial of $\mathcal{C}(X)$ (Theorem
\ref{Mch}) and to the Poincar\'{e} polynomial of the complement
$\rx$ of the toric arrangement (Theorem \ref{MPo}). In
particular $M_X(1,0)$ is equal to the Euler characteristic of $\rx$,
and also to the number of connected components of the complement
of the arrangement in the \emph{compact torus}
$\overline{T}=(\mathbb{S}^1)^n$.

We call $M_X(x,y)$ the \emph{multiplicity Tutte polynomial} of
$X$, since it coincides with $T_X(x,y)$ when X is unimodular and, in general, it satisfies the same
\emph{deletion-restriction} recursion that holds for $T_X(x,y)$. By
this formula (Theorem \ref{rec}) we prove that $M_X(x,y)$ has
positive coefficients (Theorem \ref{pos}).
We leave open the problem of explaining the meaning of these coefficients (Problem \ref{IntComb}).

A similar polynomial can be defined more generally for
matroids, if we enrich their structure in order to encode some
"arithmetic data"; we call such objects \emph{multiplicity
matroids}. We hope to develop in a future paper an axiomatic theory of these matroids, as well as applications to graph theory, which are only outlined here. The focus of the present paper is in the case we have a list $X$ of vectors
with integer coordinates.

Given such a list, we consider two finite dimensional vector
spaces: a space of polynomials $D(X)$ defined by differential
equations, and a space of quasipolynomials $DM(X)$ defined by
difference equations. These spaces were introduced by Dahmen
and Micchelli to study \emph{box splines} and
\emph{partition functions}, and are deeply related respectively
with the hyperplane arrangement and the toric arrangement
defined by $X$, as explained in the recent book \cite{li}.
In particular, $T_X(1,y)$ is known to be the Hilbert series of
$D(X)$, and we will show that $M_X(1,y)$ is the Hilbert series of
$DM(X)$ (Theorem \ref{dDM}).

On the other hand, by Theorem \ref{inZ} the coefficients of $M_X(x,1)$ count integer
points in some faces of a convex polytope, the \emph{zonotope} 
defined by $X$. The relations between zonotopes and Dahmen-Micchelli spaces is being studied intensively over the last years, giving arise to many algebraic and combinatorial constructions (see in particular \cite{li}, \cite{HR}, \cite{DPV1}, \cite{AP}, \cite{HRX}, \cite{Le}). In particular $M_X(1,1)$
is equal to the volume of the zonotope  (Proposition \ref{voZ}), while $M_X(2,1)$ is the number of its integer points (Proposition \ref{M21}).

Finally, we focus on the case in which $X$ is a root system. We will then show some connections with the theory of Weyl groups (see for instance Corollary \ref{wey}).

\paragraph{Acknowledgments.}
I am grateful to my advisor Corrado De Concini for many important suggestions. I also wish to thank Federico Ardila, Michele D'Adderio, Emanuele Delucchi, Matthias Lenz, Mario Marietti, Claudio Procesi, Mimi Tsuruga and Michele Vergne for stimulating discussions and remarks.

\section{Multiplicity matroids and \\multiplicity Tutte polynomials}

\subsection{Definitions}
We start by recalling the notions we will be generalizing.

A \emph{matroid} $\mathfrak{M}$ is a pair $(X, I)$, where $X$
is a finite set and $I$ is a family of subsets of $X$ (called
the \emph{independent sets}) with the following properties:
\begin{enumerate}
  \item The empty set is independent;
  \item every subset of an independent set is independent;
  \item let $A$ and $B$ be two independent sets and assume
      that $A$ has more elements than B. Then there exists
      an element $a\in A\setminus B$ such that $B\cup\{a\}$
      is still independent.
\end{enumerate}

A maximal independent set is called a \emph{basis}. The last
axiom implies that all bases have the same cardinality, which is called
the \emph{rank} of the matroid. Every $A\subseteq X$ has a
natural structure of matroid, defined by considering a subset
of $A$ independent if and only if it is in $I$. Then each
$A\subseteq X$ has a rank which we denote by $r(A)$.

The \emph{Tutte polynomial} of the matroid is then defined as
$$T(x,y)\doteq \sum_{A\subseteq X} (x-1)^{r(X)-r(A)} (y-1)^{|A|-r(A)}.$$

From the definition it is clear that $T(1,1)$ is equal to the number
of bases of the matroid.

The interested reader may refer for instance to \cite{Ox}, \cite{li} 

In the next sections we will recall the two most important examples of matroid
and some properties of their Tutte polynomials.

~

We now introduce the following definitions.

A \emph{multiplicity matroid} $\mathfrak{M}$ is a triple $(X,
I, m)$, where $(X, I)$ is a matroid and $m$ is a function
(called \emph{multiplicity}) from the family of all subsets of
$X$ to the positive integers.

 We say that $m$ is the \emph{trivial multiplicity} if it is identically equal to $1$.

We define the \emph{multiplicity Tutte polynomial} of a
multiplicity matroid as
$$M(x,y)\doteq \sum_{A\subseteq X} m(A) (x-1)^{r(X)-r(A)} (y-1)^{|A|-r(A)}.$$

Let us remark that we can endow every matroid with the trivial
multiplicity, and then $M(x,y)=T(x,y)$.

\begin{re}\label{sum}
Given any two matroids $\mathfrak{M}_1=(X_1, I_1)$ and
$\mathfrak{M}_2=(X_2, I_2)$, we can naturally define a matroid
$\mathfrak{M}_1\oplus \mathfrak{M}_2=(X,I)$, where $X$ is the
disjoint union of $X_1$ and $X_2$, and $A\in I$ if and only if
$A_1\doteq A\cap X_1\in I_1$ and $A_2\doteq A\cap X_2\in I_2$.
Moreover, if $\mathfrak{M}_1$ and $\mathfrak{M}_2$ have
multiplicity functions $m_1$ and $m_2$, $m(A)\doteq
m_1(A_1)\cdot m_2(A_2)$ defines a multiplicity on
$\mathfrak{M}_1\oplus \mathfrak{M}_2$. We notice that the rank
of a subset $A$ is just the sum of the ranks of $A_1$ and
$A_2$, and so it is easily seen that the (multiplicity) Tutte
polynomial of $\mathfrak{M}_1\oplus \mathfrak{M}_2$ is the
product of the (multiplicity) Tutte polynomials of
$\mathfrak{M}_1$ and $\mathfrak{M}_2$.
\end{re}

\subsection{Lists of vectors and zonotopes}

Let $X$ be a finite list of vectors spanning a real vector space $U$, and
$I$ be the family of its linearly independent subsets. Then
$(X, I)$ is a matroid, and the rank of a subset $A$ is just the
dimension of the spanned subspace. We denote by $T_X(x,y)$ the
associated Tutte polynomial.

We associate to the list $X$ a \emph{zonotope}, that is a convex polytope in $U$ defined as follows:
$$\mathcal{Z}(X)\doteq \left\{\sum_{x\in X}t_x x, 0\leq t_x\leq 1\right\}.$$
Zonotopes play an important role in the theory of hyperplane
arrangements, and also in that of \emph{splines}, a class of functions studied in Approximation Theory. (see \cite{li}).

~

 We recall that a \emph{lattice} $\Lambda$ of rank $n$ is a discrete subgroup of $\mathbb{R}^n$ which spans $\mathbb{R}^n$ as a real vector space.
 Every such $\Lambda$ can be generated from some basis of the vector space by forming all linear combinations with integer coefficients;
 hence the group $\Lambda$ is isomorphic to $\mathbb{Z}^n$.
 We will use always the term \emph{lattice} with this meaning, and not in the combinatorial sense (a poset with \emph{join} and \emph{meet}).

Now let $X$ be a finite list of elements in a lattice $\Lambda$,
and let $I$ and $r$ be as above. For every $A\subset X$, we denote by $\langle A\rangle_{\mathbb{Z}}$  and $\langle A\rangle_{\mathbb{R}}$ respectively the sublattice of $\Lambda$ and the subspace of $U\doteq \Lambda\otimes \mathbb{R}$ spanned by $A$. Let us define $$\Lambda_A\doteq \Lambda\cap \langle A\rangle_{\mathbb{R}},$$
the largest sublattice of $\Lambda$ in which $\langle A\rangle_{\mathbb{Z}}$ has finite index. We define $m$ as this index:
$$m(A)\doteq \left[\Lambda_A : \langle A\rangle_{\mathbb{Z}}\right].$$

This defines a multiplicity matroid
and then a multiplicity Tutte polynomial $M_X(x,y)$, which is
the main subject of this paper.

Concretely, if we identify $\Lambda$ with $\mathbb{Z}^n$, we have that for every $A\subset X$ of maximal rank, $m(A)$ is equal to the GCD of the determinants of the bases extracted from $A$.

We say that the list $X$ is \emph{unimodular} if every $B\subset X$ that is a basis for $U$ over $\mathbb{R}$ spans $\Lambda$ over $\mathbb{Z}$ (i.e, in the identification above, $B$ has determinant $\pm 1$). In this case the multiplicity is trivial and $M_X(x,y)=T_X(x,y)$.

We now start by showing
some relations with the zonotope $\mathcal{Z}(X)$
generated by $X$ in $U$.

We already observed that $T_X(1,1)$ is equal to the number of bases that can be
extracted from $X$. On the other hand we have

\begin{pr}\label{voZ}
$M_X(1,1)$ is equal to the volume of the zonotope $\mathcal{Z}(X)$.
\end{pr}
\begin{proof}
By \cite{Sh}, $\mathcal{Z}(X)$ is paved by a family of
polytopes $\{\Pi_B\}$, where $B$ varies among all the bases
extracted from $X$, and
every $\Pi_B$ is obtained by translating the zonotope $\mathcal{Z}(B)$ generated by the sublist $B$. Hence
$$vol(\Pi_B)=|det(B)|.$$
However, when $B$ is a basis,
\begin{equation}
    m(B)=\left[\Lambda: \langle B\rangle_{\mathbb{Z}}\right]=|det(B)|.
\end{equation}
Since
$$M_X(1,1)=\sum_{B\subset X, B \text{basis}} m(B)$$
the claim follows.
\end{proof}

Further relations between the polynomial $M_X(x,y)$ and the zonotope $\mathcal{Z}(X)$ will be shown in Section \ref{IZ}.

\subsection{Graphs}

Let $G$ be a finite graph, $V(G)$ be the set of its vertices, and $X$ be the set of its edges. For every $A\subseteq X$,  we consider the subgraph of $G$ having set of vertices $V(G)$ and set of edges $A$. By abuse of notation, we denote this subgraph by $A$. We define $I$ as the set of the \emph{forests} in $G$ ( subgraphs whose connected components are simply connected). Then $(X, I)$ is a matroid with rank function
$$r(A)=|V(G)|-c(A),$$
where $c(A)$ is the number of connected components of $A$.

\begin{re}
 If $G$ has neither loops nor multiple edges, let us take a vector space $\widetilde{U}$ with basis $e_1, \dots, e_n$ bijective to $V(G)$, and associate to the edge connecting two vertices $i$ and $j$ the vector $e_i-e_j$.
In this way we get a list $X_G$ of vectors bijective to $X$ and spanning a hyperplane $U$ in $\widetilde{U}$. Since in this bijection the rank is preserved
and forests correspond to linearly independent sets, $G$ and
$X_G$ define the same matroid and have the same Tutte
polynomial.
\end{re}

Now let us assume every edge $e\in X$ to have an integer label $m_e>0$. By defining
$$m(A)\doteq \prod_{e\in A}m_e,$$
we get a multiplicity matroid and then a multiplicity Tutte
polynomial $M_G(x,y)$.

We may view the labels $m_e$ as multiplicities
of the edges in the following way. Let us define a new graph $G_m$ with the same
vertices as $G$, but with $m_e$ edges connecting the two vertices
incident to $e\in X$. Now let $S(G_m)$ be the set of
\emph{simple} subgraphs of $G_m$, i.e subgraphs with at most one
edge connecting any two vertices, and at most one loop on every vertex.
It is then clear that
$$M_G(x,y)\doteq \sum_{A\in S(G_m)} (x-1)^{r(X)-r(A)} (y-1)^{|A|-r(A)}.$$

In particular, $M_G(2,1)$ is equal to the number of forests of $G_m$, and $M_G(1,1)$ the number of \emph{spanning trees}
(i.e., trees connecting all the vertices) of $G_m$.

\section{Deletion-restriction formula and positivity}

The central idea that inspired Tutte in defining the
polynomial $T(x,y)$ was to find the most general invariant satisfying a recursion
 known as \emph{deletion-restriction}.
Such a recursion allows to reduce the computation of the Tutte polynomial to some trivial cases.

In the two examples above (when the matroid is
defined by a list of vectors or by a graph), the polynomial $M(x,y)$ satisfies a
similar recursion. We will explain and prove the algorithm in the first case. The case of graphs
(where the recursion is known as \emph{deletion-contraction}) is described in \cite[Section 2.3.1]{te}.

\subsection{Lists of vectors}

Let $X$ be a finite list of elements spanning a vector space
$U$, and let $\lambda\in\ X$ be a nonzero element. We define two new lists: the list $X_1\doteq X\setminus\{\lambda\}$
of elements of $U$
 and the list $X_2$ of elements of $U/\langle \lambda\rangle$ obtained
by reducing $X_1$ modulo $\lambda$.
Assume that $\lambda$ is dependent in $X$, i.e. $\lambda\in \langle X_1\rangle_\mathbb{R}$.
We have the following well-known formula:
\begin{te}\label{reT}
$$T_X(x,y)=T_{X_1}(x,y)+T_{X_2}(x,y)$$
\end{te}

It is now clear why we defined $X$ as a list, and not as a set: even if we start with $X$ composed by distinct (nonzero)  elements,
 some vectors in $X_2$ may appear many times (and some
vectors may be zero).

Notice that by applying the above
formula recursively, our problem reduces to computing $T_Y(x,y)$ when
$Y$ is the union of a list $Y_1$ of $k$ linearly independent vectors and of a list $Y_0$ of $h$ zero vectors ($k,h\geq 0$).
In this case the Tutte polynomial is easily computed.
\begin{lem}
$$T_Y(x,y)=x^k y^h.$$
\end{lem}

\begin{proof}
Given any $\lambda\in Y_1$, since
$$\big\langle Y \big\rangle_{\mathbb{R}} =\big\langle Y\setminus \{v\} \big\rangle_{\mathbb{R}}\oplus \big\langle  \{v\} \big\rangle_{\mathbb{R}}$$
by Remark \ref{sum} we have that
$$T_Y(x,y)=x\: T_{Y\setminus \{\lambda\}}(x,y).$$
Hence by induction we get that $T_Y=x^k\: T_{Y_0}$.
Finally
$$T_{Y_0}(x,y)=\sum_{j=0}^{h}{h \choose j}(y-1)^j=\big((y-1)+1\big)^{h}=y^{h}.$$
\end{proof}

Thus we get:\label{poT}
\begin{te}
$T_X(x,y)$ is a polynomial with positive coefficients.
\end{te}

\subsection{Lists of elements in abelian groups.}\label{DRs}

We now want to show a similar recursion for the polynomial
$M_X(x,y)$. Inspired by \cite{DPV2}, we noticed that we need to work in a larger category. Indeed, whereas the quotient of a vector space by a subspace is still a vector space, the quotient of a lattice by a sublattice is not a lattice, but a \emph{finitely generated abelian group}. For example, in the 1-dimensional case, the quotient of $\mathbb{Z}$ by $m\mathbb{Z}$ is the cyclic group of order $m$.

Then let $\Gamma$  be a finitely generated abelian group. For every subset $S$ of $\Gamma$ we denote by $\langle S\rangle$ the generated subgroup. We recall that $\Gamma$ is isomorphic to the direct product of a lattice $\Lambda$ and of a finite group $\Gamma_t$, which is called the \emph{torsion subgroup} of $\Gamma$. We denote by $\pi$ the projection $\pi:\Gamma\rightarrow\Lambda$.

Let $X$ be a finite subset of $\Gamma$. For every $A\subseteq X$ we set
$$\Lambda_A\doteq\Lambda\cap\big\langle\pi(A)\big\rangle_\mathbb{R}$$
and
$$\Gamma_A\doteq\Lambda_A\times \Gamma_t.$$
In other words, $\Gamma_A$ is the largest subgroup of $\Gamma$ in which $\langle A \rangle$ has finite index.

\medskip

Now we define

$$m(A)\doteq \big[ \Gamma_A : \langle A \rangle \big].$$

We also define $r(A)$ as the rank of $\pi(A)$.
In this way we defined a multiplicity matroid, to which we associate its multiplicity Tutte polynomial:
$$M_X(x,y)\doteq \sum_{A\subseteq X} m(A) (x-1)^{r(X)-r(A)} (y-1)^{|A|-r(A)}.$$
It is clear that if $\Gamma$ is a lattice, these definitions coincide with those given in the previous sections.

If on the opposite case, $\Gamma$ is a finite group, then $M(x,y)$ is a
polynomial in which only the variable $y$ appears. Furthermore, this
polynomial evaluated at $y=1$, gives the order of $\Gamma$.
Indeed the only summand that does not vanish is the contribution of
the empty set, which generates the trivial subgroup.

Now let $\lambda\in\ X$ be a nonzero element such that
\begin{equation}
\pi(\lambda)\in\left\langle\pi\big(X\setminus\{\lambda\}\big)\right\rangle_\mathbb{R}
\end{equation}

We set $$X_1\doteq X\setminus\{\lambda\}\subset \Gamma.$$

For every $A\subseteq X$, we denote by $\overline{A}$ its image under the natural projection
$$\Gamma\longrightarrow \Gamma/\langle\lambda\rangle.$$
Since $\Gamma/\langle\lambda\rangle$ is a finitely generated abelian group and $\overline{A}$ is a subset of it, $m(\overline{A})$ is defined.
Notice that
$$m(\overline{A})\doteq \big[ (\Gamma/\langle\lambda\rangle)_{\overline{A}} : \langle \overline{A} \rangle \big]=
 \big[ \Gamma_A/\langle\lambda\rangle : \langle A \rangle/\langle\lambda\rangle \big]=
 \big[ \Gamma_A : \langle A \rangle \big]=m(A).$$
We denote by $X_2$ the subset $\overline{X_1}$ of $\Gamma/\langle\lambda\rangle$. We have the following deletion-restriction formula.
\begin{te}\label{rec}
$$M_X(x,y)=M_{X_1}(x,y)+M_{X_2}(x,y).$$
\end{te}
\begin{proof}
The sum expressing $M_X(x,y)$ splits into two parts. The first is over the sets $A\subseteq X_1$
$$\sum_{A\subseteq X_1} m(A) (x-1)^{r(X)-r(A)} (y-1)^{|A|-r(A)}=M_{X_1}(x,y)$$
since clearly $r(X)=r(X_1)$. The second part is over the sets $A$ such that $\lambda\in A$. For such sets we have that
$$|\overline{A}|=|A|-1,\; r(\overline{A})=r(A)-1,\; r(X_2)=r(X)-1,\; m(\overline{A})=m(A).$$
Therefore
$$\sum_{A\subseteq X, \lambda\in A} m(A) (x-1)^{r(X)-r(A)} (y-1)^{|A|-r(A)}=$$
$$\sum_{\overline{A}\subseteq X_2} m(\overline{A}) (x-1)^{r(X_2)-r(\overline{A})} (y-1)^{|\overline{A}|-r(\overline{A})}=M_{X_2}(x,y).$$
\end{proof}


Now we can prove
\begin{te}\label{pos}
$M_X(x,y)$ is a polynomial with positive coefficients.
\end{te}

\begin{proof}
By applying recursively the formula above, we need only consider lists that do not contain any $\lambda$ satisfying condition (2).
Any such list $Y$ is made of elements of some quotient $\Gamma(Y)$ of $\Gamma$, and is the disjoint union of a list $Y_0$ of $h$ zeros $(h\geq 0)$, and of a list $Y_1$ such that $\pi(Y_1)$ is a basis of $\Lambda(Y)\otimes \mathbb{R}$.
(Here we denote by $\pi$ the projection $\Gamma(Y)\rightarrow \Lambda(Y)$, where $\Gamma(Y)\simeq\Lambda(Y)\times \Gamma(Y)_t$ is the product of the lattice and of the torsion subgroup). We first notice that
$$M_{Y_0}=|\Gamma(Y)_t|\sum_{j=0}^{h}{h \choose j}(y-1)^j=|\Gamma(Y)_t|\big((y-1)+1\big)^{h}=|\Gamma(Y)_t|y^{h}.$$
Furthermore, it is easily seen that
$$M_Y(x,y)=M_{Y_0}(x,y) M_{Y_1}(x,y).$$
Finally, the positivity of $M_{Y_1}(x,y)$ will be proved in Lemma \ref{Zba}.
\end{proof}

\subsection{Statistics}

Usually, polynomials with positive coefficients encode some statistics. In other words, their coefficients \emph{count} something.

For example, the Tutte polynomial embodies two statistics on the set of the bases called internal and external activity. Although they can be stated for an abstract matroid (see for example \cite[Section 2.2.2]{li}), we give such definitions for a list $X$ of vectors. Let $B$ be a basis extracted from $X$.
 \begin{enumerate}
   \item We say that $v\in X\setminus B$ is \emph{externally active} for $B$ if $v$ is a linear combination of the elements of $B$ following it (in the total ordering fixed on $X$);
   \item we say that $v\in B$ is \emph{internally active} for $B$ if there is no element $w$ in $X$ preceeding $v$ such that $\{w\}\cup(B\setminus \{v\})$ is a basis.
   \item the number $e(B)$ of externally active elements is called the \emph{external activity} of $B$;
   \item the number $i(B)$ of internally active elements is called the \emph{internal activity} of $B$.
  \end{enumerate}

The following result is proved in \cite{Cr}:
\begin{te}\label{Cra}
$$T_X(x,y)=\sum_{B\subseteq X,\: B \text{basis}}x^{i(B)}y^{e(B)}.$$
\end{te}

Hence the coefficients of $T_X(x,y)$ count the number of bases having a given internal and external activity.

Since $M_X(x,y)$ has positive coefficients too, it is natural to wonder what the statistics involved are.
\begin{prob}\label{IntComb}
Give a combinatorial interpretation of the coefficients of $M_X(x,y)$.
\end{prob}
Although we leave this question open, in Theorems \ref{inZ} and \ref{dDM} we show the meaning of the coefficients of $M_X(x,1)$ and $M_X(1,y)$ respectively.

We say that a basis $B$ of $X$ is a \emph{no-broken circuit} if $e(B)=0$. We denote
by $nbc(X)$ the number of no-broken circuit bases of $X$. It is clear from Theorem \ref{Cra} that
\begin{equation}
nbc(X)=T_{X}(1,0).
\end{equation}
We will use this formula in the following sections.

\section{Integer points in zonotopes}\label{IZ}

Let $X$ be a finite list of vectors contained in a lattice $\Lambda$ and generating the vector space $U=\Lambda\otimes \mathbb{R}$.
We say that a point of $U$ is \emph{integer} if it is contained in $\Lambda$. In this section we prove that $M_X(2,1)$ is equal to the number of integer points of the zonotope $\mathcal{Z}(X)$. Moreover, we compare this number with the volume. To do this, we have to move the zonotope to a "generic position". 

Following \cite[Section 1.3]{li}, we define the \emph{cut-locus} of the couple $(\Lambda, X)$ as the union of all hyperplanes in $U$ that are translations, under elements of $\Lambda$, of the linear hyperplanes spanned by subsets of $X$. Let $\underline{\varepsilon}$ be a vector of $U$ that does not lie in the cut-locus and has length $\varepsilon << 1$. Let $\mathcal{Z}(X)-\underline{\varepsilon}$ be the polytope obtained by translating $\mathcal{Z}(X)$ by $-\underline{\varepsilon}$, and let $\mathfrak{I}(X)$ be the set of its integer points:
$$\mathfrak{I}(X)\doteq \left(\mathcal{Z}(X)-\underline{\varepsilon}\right)\cap\Lambda.$$
It is intuitive (and proved in \cite[Prop 2.50]{li}) that this number is equal to the volume:
$$\left|\mathfrak{I}(X)\right|=vol\left(\mathcal{Z}(X)\right)=M_X(1,1)$$
by Proposition \ref{voZ}. We now prove a stronger result. 

Let us
choose $\underline{\varepsilon}$ so that $\mathcal{Z}(X)-\underline{\varepsilon}$ contains the origin $\underline{0}$. We partition
$\mathfrak{I}(X)$ as follows: set $\mathfrak{I}_n(X)=\{\underline{0}\}$, and for every $k=n-1, \ldots, 0$, let $\mathfrak{I}_k (X)$ be the set of elements of $\mathfrak{I} (X)$ that are contained in some $k-$codimensional face of $\mathcal{Z}(X)$ and that are not contained in $\mathfrak{I}_{h} (X)$ for $h>k$.
Then we have:

\begin{te}\label{inZ}
$$M_X(x,1)=\sum_{k=0}^n \left| \mathfrak{I}_k(X)\right| \,x^k.$$
\end{te}

\begin{ex}\label{eZo}
Consider the list in $\mathbb{Z}^2$
$$X=\left\{ (3,3), (1,-1), (2,0)\right\}.$$
Then
$$M_X(x,y)=(x-1)^2+(3+1+2)(x-1)+(6+6+2)+2(y-1).$$
Hence $$M_X(x,1)=x^2+4x+9$$
and $M_X(2,1)=21$. Indeed the zonotope $\mathcal{Z}(X)$ has area $14$ and contains $21$ integer points, $14$ of which lie in $\mathcal{Z}(X)-\underline{\varepsilon}$, represented by the shaded portion in the image below. The sets $\mathfrak{I}_2(X)$, $\mathfrak{I}_1(X)$, and $\mathfrak{I}_0(X)$ contain 1, 4 and 9 points, respectively, each marked by the subscript of its set.

\includegraphics[width=80mm]{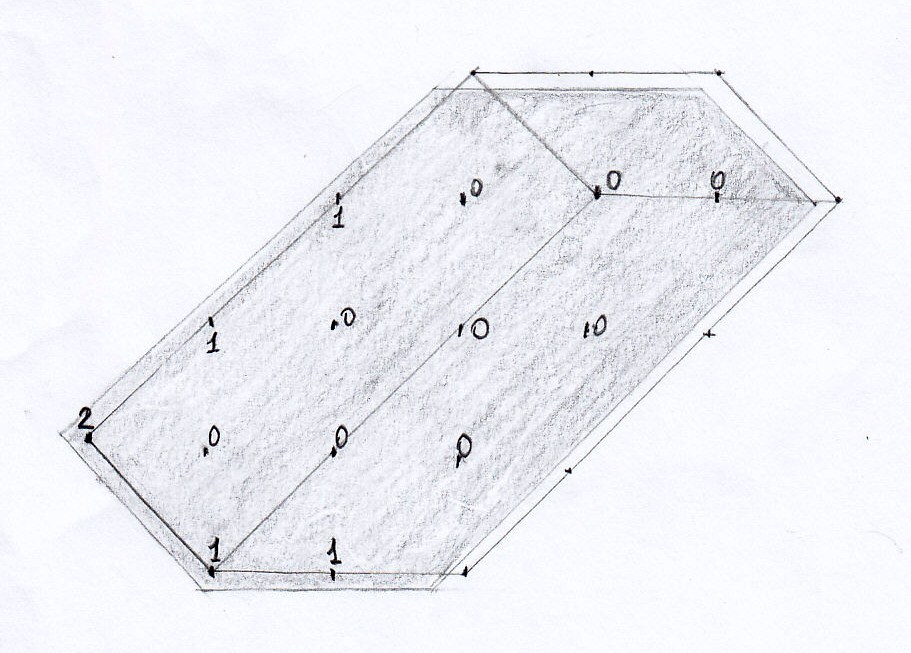}

\end{ex}

In order to prove the theorem above, we first assume $X$ to be linearly independent (and hence a basis for $U$).
In this case the zonotope is called a \emph{parallelepiped}. In particular, for every $A\subseteq X$ the zonotope $\mathcal{Z}(A)$ is a face of $\mathcal{Z}(X)$.
Moreover, we can choose $\underline{\varepsilon}$ so that the faces of $\mathcal{Z}(X)$ whose interior is contained in $\mathcal{Z}(X)-\underline{\varepsilon}$ are precisely those of type $\mathcal{Z}(A)$ for some $A\subseteq X$: for instance take
 $$\underline{\varepsilon}=\sum_{\lambda\in X}\frac{\varepsilon}{n} \lambda.$$

We say that an integer point is \emph{internal} to a face of $\mathcal{Z}(X)$ if that face is the smallest face containing this point.
 We denote by $h(A)$ the number of integer points that are internal to $\mathcal{Z}(A)$.

\begin{lem}
For every $A\subset X$,
$$h(A)=\sum_{B\subseteq A}(-1)^{|A|-|B|}m(B).$$
\end{lem}

\begin{proof}
By construction $\mathcal{Z}(X)-\underline{\varepsilon}$ contains exactly the integer points that are internal to the faces $\mathcal{Z}(A),\: A\subseteq X$. Hence
$$\left|\mathfrak{I}(X)\right|=\sum_{A\subseteq X}h(A).$$
Moreover, by Formula (1) $m(X)$ is equal to the volume of $\mathcal{Z}(X)$. Thus we proved:
$$m(X)=\sum_{A\subseteq X}h(A).$$
We get the claim by the inclusion-exclusion principle, since the intersection of two faces $\mathcal{Z}(A_1)$, $\mathcal{Z}(A_2)$ is the face
$\mathcal{Z}(A_1\cap A_2)$.
\end{proof}

Now we prove

\begin{lem}\label{iZi}\label{Zba}
Let $X$ be a basis for $U$. Then
$$M_X(x,1)=\sum_{k=0}^n\left( \sum_{A\subseteq X,\: |A|=n-k}h(A)  \right) \,x^k.$$
\end{lem}

\begin{proof}
By definition
$$M_X(x,y)= \sum_{A\subseteq X} m(A) (x-1)^{n-|A|}.$$
The coefficient of $x^k$ in this expression is
$$\sum_{A\subseteq X,\: |A|\leq n-k}(-1)^{n-k-|A|} {{n-|A|} \choose k} \,m(A).$$
By the previous Lemma, our claim amounts to proving that the coefficient of $x^k$ is
$$ \sum_{A\subseteq X,\: |A|=n-k} \sum_{B\subseteq A}(-1)^{|A|-|B|}m(B)=\sum_{B\subseteq X,\: |B|\leq n-k}(-1)^{n-k-|B|} {{n-|B|} \choose k} \,m(B)$$
because every $B\subseteq X$ is contained in exactly
$${{n-|B|} \choose {n-k-|B|}}={{n-|B|} \choose k}$$
 sets $A\subseteq X$ of cardinality $n-k$.
\end{proof}

In this way the theorem is proved for linearly independent set of vectors, since
$$|\mathfrak{I}_k(X)|=\sum_{A\subseteq X,\: |A|=n-k}h(A).$$
As in Section \ref{DRs}, given a nonzero element $\lambda\in X$ we set $X_1\doteq X\setminus\{\lambda\}$
and we denote by $X_2$ the image of $X_1$ under the natural projection
$$p: U\longrightarrow U/\langle\lambda\rangle_{\mathbb{R}}.$$
Set
$m_\lambda\doteq m\big(\{\lambda\}\big)=\left[\Lambda\cap \langle \lambda\rangle_{\mathbb{R}} : \langle \lambda\rangle_{\mathbb{Z}}\right].$
 Note that $X_1$ defines a zonotope $\mathcal{Z}(X_1)\subset \mathcal{Z}(X)$, and $X_2$ defines a zonotope $\mathcal{Z}(X_2)$ in the quotient space $p(U)$.
 We briefly recall some results from \cite[Section 2.3]{li}. The reader is suggested to look at the picture below (in which $\mathcal{Z}(X_1)$ is the white rectangle, and $\lambda$ is the vector of coordinates $(2,0)$).
 The 1-parameter group of translations $u\mapsto u+t\lambda$, acting on $U$, has orbits that are the fibers of $p$. For every $u\in\mathcal{Z}(X_2)$, let $s(u)$ be the point $u+t\lambda$ where the translation group "exits" the zonotope $Z(X_1)$. Then $s(\mathcal{Z}(X_2))$ (the bold line in the picture) is a piece of the boundary of $\mathcal{Z}(X_1)$ naturally identified to $\mathcal{Z}(X_2)$. Furthermore, we have the following decomposition:
$$\mathcal{Z}(X)=\mathcal{Z}(X_1)\cup\widetilde{\mathcal{Z}}(X_2)$$
where $\mathcal{Z}(X_1)\cap\widetilde{\mathcal{Z}}(X_2)=s(\mathcal{Z}(X_2))$,
and the polytope $\widetilde{\mathcal{Z}}(X_2)$ (shaded in the picture) is a product of the polytope $s(\mathcal{Z}(X_2))$ and the segment $[0, \: m_\lambda]$. So $p$ maps $\widetilde{\mathcal{Z}}(X_2)$ on $\mathcal{Z}(X_2)$ with fiber $[0, \: m_\lambda]$ (\cite[Section 2.3]{li}).

\includegraphics[width=80mm]{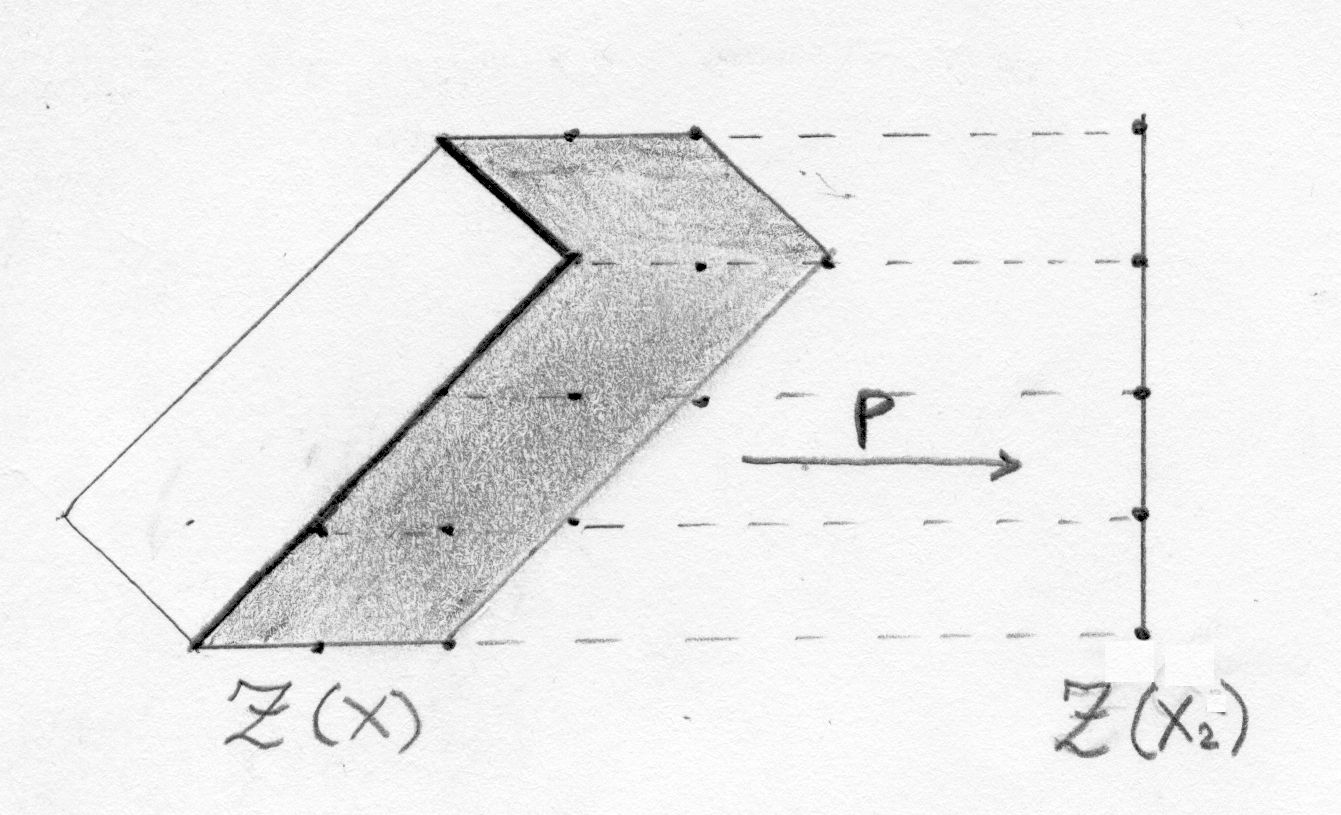}

We can now conclude the proof of Theorem \ref{inZ}.
\begin{proof}
 We can assume $X$ contains only nonzero elements. If $X$ is independent, we have Lemma \ref{iZi}. If $X$ has rank 0, the statement is trivial. Thus we assume $X$ to be dependent and of positive rank, we choose a vector $\lambda$ that lies in the cone spanned by the other vectors, and we proceed inductively by deletion-restriction. We notice that the restriction of $p$ to $\mathfrak{I}_k(X)\setminus \mathfrak{I}_k(X_1)$ maps this set onto $\mathfrak{I}_k(X_2)$, and the fiber of every point has cardinality $m_\lambda$. Thus
$$|\mathfrak{I}_k(X)|=|\mathfrak{I}_k(X_1)|+m_\lambda |\mathfrak{I}_k(X_2)|=M_{X_1}(x,1)+M_{X_2}(x,1)$$
by inductive hypothesis, since $m_\lambda$ is the torsion of $\Lambda/\langle\lambda\rangle_{\mathbb{Z}}$. Hence the claim follows by Theorem \ref{rec}.
\end{proof}

\includegraphics[width=80mm]{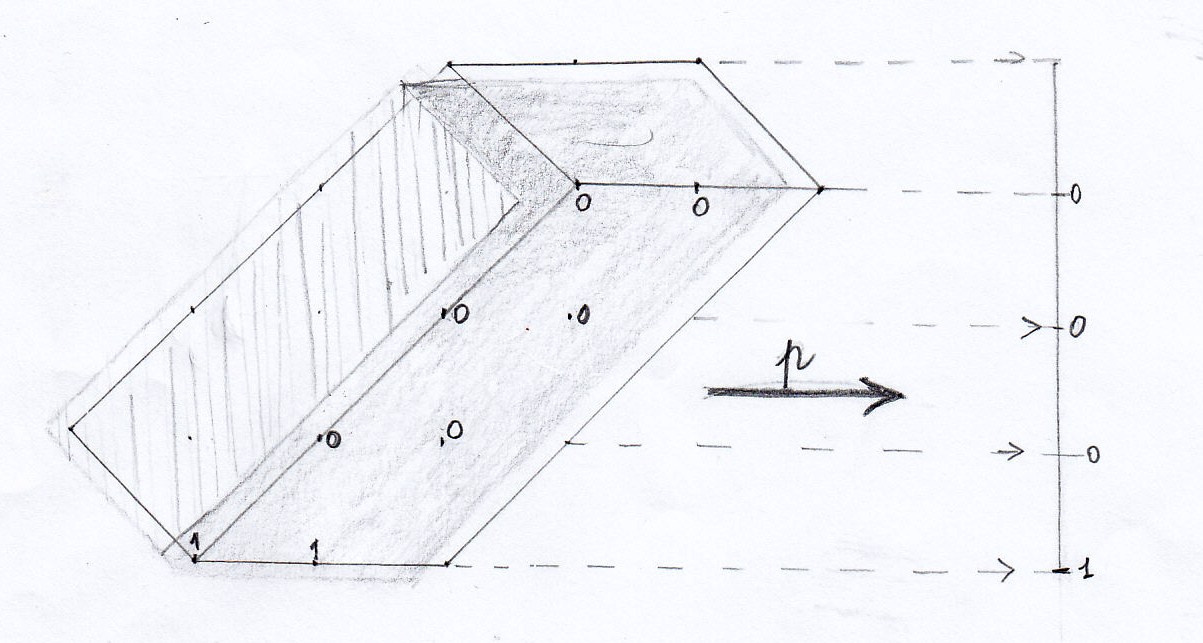}

In the same way we can prove

\begin{pr}\label{M21}
The number $\left| \mathcal{Z}(X)\cap\Lambda \right|$ of integer points in the zonotope is equal to $M_X(2,1)$.
\end{pr}
\begin{proof}
By applying deletion-restriction as in the previous proof, we can reduce to the case in which $X$ is a basis of $U$. Then in this basis, $\mathcal{Z}(X)$ is a parallelepiped. For every face $F$ we define $A_F$ as the subset of $X$ corresponding to the coordinates which are not constant on $F$. Since all the other coordinates are identically equal either to $0$ or to $1$, for every $A\subseteq X$ there are exactly $2^k$ faces $F$ such that $A_F=A$, $k= |X\setminus A|$. Among these faces, the only one contributing to $M_X(1,1)$ is the one whose constant coordinates are all equal to 0, i. e., $Z(A)$. On the other hand, to compute the total number of integer points we have to take all these $2^k$ faces. Since any two of them are disjoint and contain the same number of points in their interior, by Theorem \ref{inZ} we get the claim.
\end{proof}

\section{Application to arrangements}

In this Section we describe some geometrical objects related to
the lists considered in Section 2.2, and show that many of
their features are encoded in the polynomials $T_X(x,y)$ and
$M_X(x,y)$.

\subsection{Recall on hyperplane arrangements}

Let $X$ be a finite list of elements of a vector space $U$.
A \emph{hyperplane arrangement} $\mathcal{H}(X)$ is defined in the dual space $V=U^*$, by taking the orthogonal hyperplane of each element of
$X$. Conversely, given an arrangement of hyperplanes in a
vector space $V$, let us choose for each hyperplane a nonzero
vector in $V^*$ orthogonal to it. Let $X$ be the list of such
vectors. Since every element of $X$ is determined up to scalar
multiples, the matroid associated to $X$ is well defined. In
this way a Tutte polynomial is naturally associated to the
hyperplane arrangement.

The importance of the Tutte polynomial in the theory of
hyperplane arrangements is well known. Here we recall some
results which will be generalized in the next sections.

To every sublist $A\subseteq X$ we associate the subspace
$A^\bot$ of $V$ that is the intersection of the corresponding
hyperplanes of $\mathcal{H}(X)$. In other words, $A^\bot$ is the subspace of
vectors that are orthogonal to every element of $A$. Let
$\mathcal{L}(X)$ be the set of such subspaces, partially
ordered by reverse inclusion, and having as minimal element (which we denote by
$\mathbf{0}$) the whole space $V=\emptyset^\bot$.
$\mathcal{L}(X)$ is called the \emph{intersection poset} of the
arrangement, and is the main combinatorial object
associated to a hyperplane arrangement.

We also recall that to every finite poset $\mathcal{P}$ there is
an associated \emph{Moebius function}
$$\mu:\mathcal{P}\times \mathcal{P}\rightarrow \mathbb{Z}$$
recursively defined as follows:
$$\mu(L,M)=\begin{cases}
0  & \mbox{ if } L>M\\
1 &  \mbox{ if } L=M\\
- \sum_{L\leq N< M} \mu(L,N) & \mbox{ if } L<M.
\end{cases}$$

Notice that the poset $\mathcal{L}(X)$ is \emph{ranked} by the
dimension of the subspaces. We define \emph{characteristic
polynomial} of the poset as
$$\chi(q)\doteq \sum_{L\in \mathcal{L}(X)} \mu (\mathbf{0},L) q^{dim (L)}.$$

This is an important invariant of $\mathcal{H}(X)$.
Indeed, let
$\mathcal{M}_X$ be the complement in $V$ of the union of the
hyperplanes of $\mathcal{H}(X)$. Let $P(q)$ be \emph{Poincar\'{e}
polynomial} of $\mathcal{M}_X$, i.e. the polynomial having the $k-$th Betti number of
$\mathcal{M}_X$ as the
coefficient of $q^k$ .
If $V$ is a complex vector
space, by \cite{OSb} we have the following theorem.
\begin{te}
$$P(q)={(-q)}^n \chi(-1/q).$$
\end{te}

If, on the other hand, $V$ is a real vector space, by \cite{Za} the number $Ch(X)$ of \emph{chambers} (i.e., connected components of $\mathcal{M}_X$) is:
\begin{te}
$$Ch(X)={(-1)}^n \chi(-1).$$
\end{te}

The Tutte polynomial
$T_X(x,y)$ turns out to be a stronger invariant, in the
following sense. Assume that $\underline{0}\notin X$; then

\begin{te}
$$(-1)^n T_X(1-q,0)= \chi(q).$$
\end{te}

The proof of these theorems can be found, for example, in \cite[Theorems 10.5, 2.34 and 2.33]{li}.

Notice that in the present section we have only considered \emph{central} arrangements, i.e. arrangements in which all the hyperplanes contain the origin. More generally, one may consider arrangements of affine hyperplanes in affine spaces. In this setting, Ardila defined a notion of \emph{semimatroid} and a Tutte polynomial (see \cite{Ar}).

\subsection{Toric arrangements and their generalizations}

Let $\Gamma=\Lambda\times \Gamma_t$ be a finitely generated abelian group, and define
$$T_\Gamma\doteq {Hom(\Gamma,\mathbb{C}^*)}.$$
$T_\Gamma$ has a natural structure of abelian linear algebraic group:
indeed it is the direct product of a complex torus $T_\Lambda$ of the same rank
as $\Lambda$ and of the finite group ${\Gamma_t}^*$ dual to
${\Gamma_t}$ (and isomorphic to it).

Moreover, $\Gamma$ is identified with the
group of characters of $T_\Gamma$: indeed given $\lambda\in \Lambda$
and $t\in T_\Gamma$ we can take any representative $\varphi_t\in
Hom(\Gamma,\mathbb{C})$ of $t$ and set
$$\lambda(t)\doteq e^{2\pi i \varphi_t(\lambda)}.$$
When this is not ambiguous we will denote $T_\Gamma$ by $T$.

Let $X\subset \Lambda$ be a finite subset spanning a sublattice
of $\Lambda$ of finite index. The kernel of every character
$\chi\in X$ is a (non-connected) hypersurface in $T$:
$$H_\chi\doteq \big\{ t\in T |\chi(t)=1 \big\}.$$
The collection $\mathcal{T}(X)=\left\{H_\chi, \chi\in
X\right\}$ is called the \emph{generalized toric arrangement}
defined by $X$ on $T$.

We denote by $\rx$ the complement of the arrangement:
$$\rx\doteq T\setminus\bigcup_{\chi\in X}H_\chi$$
and by $\mathcal{C}_X$ the set of all the
connected components of all the intersections of the
hypersurfaces $H_\chi$, ordered by reverse inclusion and having
as minimal elements the connected components of $T$.

Since $rank(\Lambda)= dim(T)$, the maximal elements of
$\mathcal{C}(X)$ are 0-dimensional, hence (since they are
connected) they are points. We denote by $\mathcal{C}_0(X)$ the
set of such layers, which we call the \emph{points} of the
arrangement.

Given $A\subseteq X$ let us define
$$H_A\doteq \bigcap_{\lambda\in A}H_\lambda.$$

\begin{lem}\label{mAc}
$m(A)$ is equal to the number of connected components of $H_A$.
\end{lem}

\begin{proof}
It is clear by definition that $m(X)=|\mathcal{C}_0(X)|$.
Then for every $A\subseteq X$, we have that
$$|\mathcal{C}(A)^0|=m(A)$$
where $\mathcal{C}(A)^0$ is the set of the points of the arrangement $\mathcal{T}(A)$ defined by $A$ in $T_{\Gamma_A}$.
Now let ${H_A}^0$ be the connected component of $H_A$ that contains the identity. This is a subtorus of $T_\Gamma$, and the quotient map
$$T_\Gamma\twoheadrightarrow T_\Gamma/{H_A}^0\simeq T_{\Gamma_A}$$
induces a bijection between the connected components of $H_A$ and the points of $\mathcal{T}(A)$.
\end{proof}

In particular, when $\Gamma$ is a lattice, $T$ is a torus and
$\mathcal{T}(X)$ is called the \emph{toric arrangement} defined
by $X$. Such arrangements have been studied for example in
\cite{L2}, \cite{DPt}, \cite{Mo}, \cite{Mw}; see \cite{li} for
a complete reference. In particular, the complement $\rx$ has
been described topologically and geometrically. In this
description the poset $\mathcal{C}(X)$ plays a major role, for
many aspects analogous to that of the intersection poset for
hyperplane arrangements (see \cite{DPt}, \cite{Mw}).

We will now explain the importance in this framework of the polynomial $M_X(x,y)$ defined in Section 3.3.

\subsection{Characteristic polynomial}

Let $\mu$ be the Moebius function of $\mathcal{C}(X)$. Notice
that we have a natural rank function given by the dimension of
the layers.
For every $C\in\mathcal{C}(X)$, let $T_C$ be the connected component of $T$ that contains $C$.
We can define the characteristic polynomial of
$\mathcal{C}(X)$ as
$$\chi(q)\doteq \sum_{C\in \mathcal{C}(X)} \mu (T_C,C) q^{dim (C)}.$$
In order to relate this polynomial with $M_X(x,y)$, we prove the following fact. Let
us assume that $X$ does not contain $\underline{0}$. For every
$C\in \mathcal{C}(X)$, set
$$\mathcal{D}(C)\doteq \left\{A\subseteq X \:|\: C \mbox{ is a connected component of } H_A\right\}.$$

\begin{lem}
$$\mu (T_C,C)=\sum_{A\in\mathcal{D}(C)} (-1)^{|A|}.$$
\end{lem}

\begin{proof}
By induction on the codimension of $C$. If it is 0 or 1, the
statement is trivial. Otherwise, by the inductive hypothesis
$$\mu (T_C,C)=-\sum_{D\supsetneq C}\mu (T_C,D)=-\sum _{D\supsetneq C}\sum_{A\in \mathcal{D}(D)}(-1)^{|A|}.$$
Proving that this sum is equal to the claimed one amounts to
prove that
$$\sum _{D\supseteq C}\sum_{A\in \mathcal{D}(D)}(-1)^{|A|}=0.$$
Now, let $B$ be the largest (hence minimum with respect to reverse inclusion) element of $\mathcal{D}(C)$. Every
$A\in \mathcal{D}(D)$ for $D\supseteq C$ is a subset of $B$,
and conversely every $A\subseteq B$ is in $\mathcal{D}(D)$ for
exactly one $D\supseteq C$ (if there were two such layers $D$,
their union would be connected). Therefore
$$\sum _{D\supseteq C}\sum_{A\in \mathcal{D}(D)}(-1)^{|A|}=\sum_{A\subseteq B }(-1)^{|A|}=0$$
where the last equality is an elementary combinatorial fact,
which is checked by looking at the binomial coefficients of
$(1-1)^k$.
\end{proof}

\begin{te}\label{Mch}
$$(-1)^n M_X(1-q,0)= \chi(q)$$
\end{te}

\begin{proof}
By definition we must prove that
$$(-1)^n \sum_{A\subseteq X} m(A)(-q)^{n-r(A)}(-1)^{|A|-r(A)}=\sum_{C\in \mathcal{C}(X)} \mu(T_C,C) q^{dim C}.$$
We remark that $$dim(C)=n-r(A)\:\forall A\in \mathcal{D}(C)$$
and $$(n-r(A))+(|A|-r(A))+n\equiv |A| \:(mod\: 2).$$ Thus we
have to prove that for every $k=0,\dots, n$,
\begin{equation}
\sum_{A\subseteq X, n-r(A)=k} m(A)(-1)^{|A|}=\sum_{C\in \mathcal{C}(X), dim(C)=k}\mu(T_C,C).
\end{equation}

By Lemma \ref{mAc}, each $A$ is in $\mathcal{D}(C)$ for exactly $m(A)$ layers $C$.
Then Formula (4) is a consequence of Lemma 4.5, since $(-1)^{|A|}$ appears $m(A)$ times in the sum.
\end{proof}

\begin{ex}\label{exa}
Take $T=(\mathbb{C}^*)^2$ with coordinates $(t,s)$ and
$$X=\left\{ (2,0), (0,2), (1,1), (1,-1)\right\}$$
defining equations:
$$t^2=1, s^2=1, ts=1, ts^{-1}=1.$$
 The hypersurfaces $H_{t^2}$ and $H_{s^2}$ have two connected components each; $H_{ts}$ and $H_{ts^{-1}}$ are connected (but their intersection is not). The  $0-$dimensional layers are
$$C_1=(1,1),\: C_2=(-1,-1),\: C_3=(1,-1),\: C_4=(-1,1).$$

Notice that $C_1$ and $C_2$ are contained in 4 layers of
dimension 1 each, while each of $C_3$ and $C_4$ lies in 2
layers of dimension 1. Then $\mu (T,C)=-1$ for each of
the six $1-$dimensional layers $C$, and
$$\mu (T,C_1)=\mu (T,C_2)=-(1-4)=3$$
$$\mu (T,C_3)=\mu (T,C_4)=-(1-2)=1.$$
Hence
$$\chi(q)=q^2-6q+8.$$
The polynomial $M_X(x,y)$ is composed by the following summands:
\begin{itemize}
  \item $(x-1)^2$, corresponding to the empty set;
  \item $6(x-1)$, corresponding to the 4 singletons, each
      giving contribution $(x-1)$ or $2(x-1)$;
  \item $14$, corresponding to the 6 pairs: indeed, the
      basis $X=\left\{ (2,0), (0,2)\right\}$ spans a sublattice of index $4$, while the other bases span sublattices of index $2$;
  \item $8(y-1)$, corresponding to the 4 triples, each
      contributing with $2(y-1)$;
  \item $2(y-1)^2$, corresponding to the whole set $X$.
\end{itemize}
Hence
$$M_X(x,y)=x^2+2y^2+4x+4y+3.$$
Notice that
$$M_X(1-q,0)=q^2-6q+8=\chi(q)$$
as claimed in Theorem \ref{Mch}.
\end{ex}

\subsection{Poincar\'{e} polynomial}

For every $C\in\mathcal{C}_X$, let us define
$$X_C\doteq \left\{\chi\in X | H_\chi\supseteq C\right\}.$$

\begin{re}
The set $X_C$ defines a hyperplane arrangement in the vector space
$V_C\doteq V/{X_C}^\bot$. Let $\mathcal{L}(X_C)$ be its
intersection poset. Let $\mathcal{C}(X,C)$ be the poset of the
elements of $\mathcal{C}(X)$ that contain $C$. The map
$$\psi: \mathcal{C}(X,C)\rightarrow\mathcal{L}(X_C)$$
$$D\mapsto {X_D}^\bot$$
is an order-preserving bijection. Indeed, given
$L\in\mathcal{L}(X_C)$, set
$$A(L)\doteq \big\{\lambda\in X, \lambda|_L=0\big\}.$$
Then $\psi^{-1}(L)$ is the connected component containing $C$ of $H_{A(L)}$.
\end{re}

\begin{lem}
$$nbc(X_C)=(-1)^{n-dim(C)}\mu(T_C,C).$$
\end{lem}
\begin{proof}
By the previous remark,
$$\mu(T_C,C)\doteq\mu_{\mathcal{C}(X)}(T_C,C)=\mu_{\mathcal{C}(X,C)}(T_C,C)=\mu_{\mathcal{L}(X_C)}(V_C,{X_C}^\bot)=\chi_{\mathcal{L}(X_C)}(0)$$
since ${X_C}^\bot$ is the origin in $V_C$, and hence the only
element of rank 0. Thus by Theorem 4.3 and Formula (3),
$$\chi_{\mathcal{L}(X_C)}(0)=(-1)^{n-dim(C)}T_{X_C}(1,0)=(-1)^{n-dim(C)}nbc(X_C).$$
\end{proof}

Let $T_1,\dots,T_h$ be the connected components of $T$. We
denote by $\mathcal{C}(X)_i$ the set of layers that are
contained in $T_i$. This clearly gives a partition of the
layers:
$$\mathcal{C}(X)_=\bigsqcup_{i=1}^{h} \mathcal{C}(X)_i.$$

We now give some formulae for the Poincar\'{e} polynomial $P(q)$
and the Euler characteristic of $\rx$. We start from a
restatement of a result proved in \cite[Theor. 4.2]{DPt} (see
also \cite[14.1.5]{li}).  In this paper is considered an
arrangement of hypersurfaces in a torus, in which every
hypersurface is obtained by translating by an element of the torus
the kernel of a character. It is clear that the restriction of
the arrangement $\mathcal{T}(X)$ on every $T_i$ is an
arrangement of this kind. Then the cohomology of $\rx\cap T_i$
can be expressed as a direct sum of contributions given by the
layers of this arrangement, which are the elements of
$\mathcal{C}(X)_i$. In terms of the Poincar\'{e} polynomial
$P_i(q)$ of $\rx\cap T_i$, this expression is:
$$P_i(q)=\sum_{C\in\mathcal{C}(X)_i} nbc(X_C)
(q+1)^{dim(C)}q^{n-dim(C)}.$$
Thus the Poincar\'{e} polynomial of $\rx=\bigsqcup_i (\rx\cap T_i)$ is just the sum of these polynomials:

\begin{te}
$$P(q)=\sum_{C\in\mathcal{C}(X)} nbc(X_C) (q+1)^{dim(C)}q^{n-dim(C)}.$$
\end{te}

Now we prove:

\begin{te}\label{MPo}
$$P(q)=q^n M_X\left(\frac{2q+1}{q}, 0\right).$$
\end{te}
\begin{proof}
By definition, we have that
$$q^n M_X\left(\frac{2q+1}{q}, 0\right)=\sum_{A\subseteq X} m(A) (q+1)^{n-r(A)}q^{r(A)} (-1)^{|A|-r(A)}.$$
We compare this formula with the one in the previous Theorem.
We have to prove that for every $k=0,\dots,n$ the coefficient
of $(q+1)^{k}q^{n-k}$ is the same in the two expressions. In
fact by applying Formula (4) and then Lemma 4.9 we get the
claim:
$$(-1)^{n-k}\sum_{A\subseteq X, r(A)=n-k} m(A)(-1)^{|A|}=(-1)^{n-k}\sum_{C\in \mathcal{C}(X), dim(C)=k}\mu(T_C,C)=$$
$$=\sum_{C\in \mathcal{C}(X), dim(C)=k} nbc(X_C).$$
\end{proof}
 Therefore, by comparing Theorem \ref{Mch} and Theorem \ref{MPo}, we get the following formula, which relates the combinatorics of $\mathcal{C}(X)$ with the topology of $\rx$, and is the "toric" analogue of Theorem 4.1.
\begin{co}
$$P(q)={(-q)}^n \chi\left(-\frac{q+1}{q}\right).$$
\end{co}

We recall that the \emph{Euler characteristic} of a space can
be defined as the evaluation at $-1$ of its Poincar\'{e}
polynomial. Hence by Theorem \ref{MPo} we have:
\begin{co}\label{MEu}
$(-1)^n M_X(1,0)$ is equal to the Euler characteristic of $\rx$.
\end{co}

\begin{ex}
In the case described in Example \ref{exa}, Theorem \ref{MPo} (or
Corollary 4.12) implies that
$$P(q)=15q^2+8q+1$$
and hence the Euler characteristic is $$P(-1)=8=M_X(1,0).$$
\end{ex}

\subsection{Number of regions of the compact torus}
In this section we consider the compact abelian group dual to $\Gamma$
$$\overline{T}\doteq {Hom(\Gamma,\mathbb{S}^1)}$$
where we set
$$\mathbb{S}^1\doteq \{z\in \mathbb{C}\:|\: |z|=1\}\simeq \mathbb{R}/\mathbb{Z}.$$

We assume for simplicity $\Gamma$ to be a lattice: then $\overline{T}$ is a
 \emph{compact torus}, i.e. it is isomorphic to
$(\mathbb{S}^1)^n$, and every $\chi\in X$ defines a hypersurface in $\overline{T}$:
$$\overline{H_\chi}\doteq \left\{ t\in \overline{T} |\chi(t)=1 \right\}.$$
We denote by  $\overline{\mathcal{T}(X)}$ this arrangement. Clearly its poset of layers is the same as for the arrangement  $\mathcal{T}(X)$ defined in the complex torus $T$.
We denote by $\overline{\rx}$ the complement
$$\overline{\rx}\doteq \overline{T}\setminus\bigcup_{\chi\in X}\overline{H_\chi}.$$
The compact toric arrangement $\overline{\mathcal{T}(X)}$ has been studied in \cite{ERS}; in particular the number $R(X)$ of \emph{region}s (i.e. of connected components) of $\overline{\rx}$ is proved to be a specialization of the characteristic polynomial $\chi(q)$:
\begin{te}
$$R(X)=(-1)^n \chi(0).$$
\end{te}
By comparing this result with Theorem \ref{Mch} we get the following
\begin{co}
$$R(X)= M_X(1,0)$$
\end{co}

\begin{ex}
In the case of Example \ref{exa}, we can represent in the real plane with coordinates $(x,y)$ the compact torus $\overline{T}$ as the square $[0,1]\times [0,1]$ with the opposite edges identified. Then the arrangement $\overline{\mathcal{T}(X)}$ is given by the lines
$$x=0,\: x=1/2,\: y=0,\: y=1/2,\: x=-y,\: x=y.$$
These lines divide the torus in $8=\chi(0)$ regions:

\includegraphics[width=50mm]{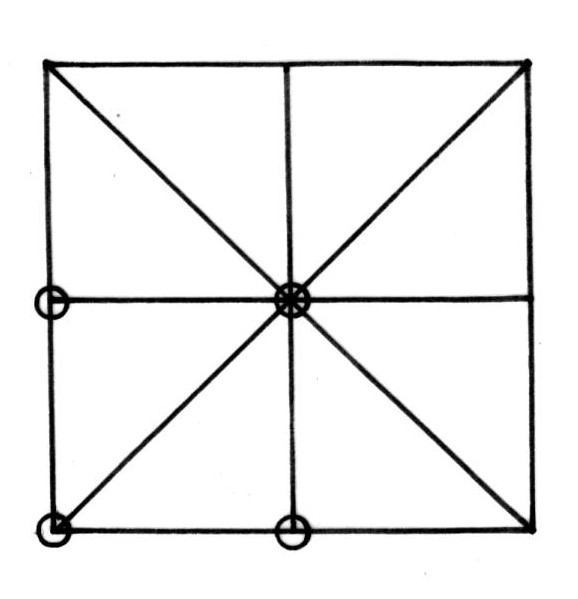}

\end{ex}

\section{Dahmen-Micchelli spaces}

So far we considered evaluations of $T_X(x,y)$
and $M_X(x,y)$ at $y=0$ and $y=1$. However,
there is another remarkable specialization of the Tutte
polynomial: $T_X(1,y)$, which (by Theorem \ref{Cra}) is called the \emph{polynomial of
the external activity} of $X$. It is related with to the
corresponding specialization of $M_X(x,y)$ in a simple way:
\begin{lem}\label{aes}
$$M_X(1,y)=\sum_{p\in \mathcal{C}_0(X)} T_{X_p}(1,y).$$
\end{lem}
\begin{proof}
By definition
$$M_X(1,y)=\sum _{A\subseteq X, r(A)=n} m(A)(y-1)^{|A|-n}$$
and
$$T_{X_p}(1,y)=\sum _{A\subseteq X_p, r(A)=n}(y-1)^{|A|-n}.$$
But by Lemma \ref{mAc}
$$m(A)=|\{p\in\mathcal{C}_0(X)| A\subseteq X_p\}|$$ which
is the number of polynomials $T_{X_p}$ in which the summand
$(y-1)^{|A|-n}$ appears.
\end{proof}

The previous lemma has an interesting consequence.
Following \cite{BH}, \cite{DM2}, and \cite{DM}, we associate
to every finite set $X\subset V$ a space $D(X)$
of functions $V\rightarrow \mathbb{C}$, and to every finite set
$X\subset \Lambda$ a space $DM(X)$ of functions
$\Lambda\rightarrow \mathbb{C}$. Such spaces are defined as the
solutions to a system, of differential equations
and of difference equations respectively, in the following way.

For every $\lambda \in X$, let $\partial_\lambda$ be  the usual directional
derivative
$$\partial_\lambda f(x)\doteq \frac{\partial f}{\partial \lambda}(x)$$
and let $\nabla_\lambda$ be the difference operator
$$\nabla_\lambda f(x)\doteq f(x)-f(x-\lambda).$$

For every $A\subset X$, we define the differential operator
$$\partial_A\doteq\prod_{\lambda\in A} \partial_\lambda$$
and the difference operator
$$\nabla_A \doteq\prod_{\lambda\in A} \nabla_\lambda.$$

We can now define the \emph{differentiable
Dahmen-Micchelli space}
$$D(X)\doteq\left\{f:V\rightarrow \mathbb{C} \:|\: \partial_A f=0\:\forall A\mbox{ such that } r(X\setminus A)<n\right\}$$
and the \emph{discrete Dahmen-Micchelli space}
$$DM(X)\doteq\left\{f:\Lambda\rightarrow \mathbb{C} \:|\: \nabla_A f =0\:\forall A\mbox{ such that } r(X\setminus A)<n\right\}.$$

The space $D(X)$ is a space of polynomials, which was introduced in order to study the \emph{box spline}. This is a piecewise-polynomial function studied in Approximation Theory; its local pieces, together with their derivatives, span $D(X)$. On the other hand, $DM(X)$ is a space of quasipolynomials which arises in the study of the \emph{partition function}. 
We recall that a function $f$ on a lattice $\Lambda$ is a \emph{quasipolynomial} if there exists a sublattice $\Lambda_0$  of finite index such that $f$ restricted to every coset coincides with a polynomial.
The partition function is the function that counts how many ways an element of $\Lambda$ can be written as a linear combination with nonnegative integer coefficients of elements of $X$. This function is a piecewise-quasipolynomial, and its local pieces, together with their translates, span $DM(X)$.

In the recent book \cite{li} the spaces $D(X)$ and $DM(X)$ are shown to be deeply related with the hyperlane arrangement and with the toric arrangement defined by $X$ respectively.

In order to compare these two spaces, we consider the elements
of $D(X)$ as functions $\Lambda\rightarrow \mathbb{C}$ by
restricting them to the lattice $\Lambda$. Since the elements
of $DM(X)$ are polynomial functions, they are determined by
their restriction. For every $p\in\mathcal{C}_0(X)$, let us
define the following map:
$$\varphi_p:\Lambda\rightarrow \mathbb{C}$$
$$\lambda\mapsto\lambda(p).$$
(see Section 2.4.2). In \cite{DM} (see also \cite[Formula
16.1]{li}) the following result is proved.

\begin{te}
$$DM(X)=\bigoplus_{p\in\mathcal{C}_0(X)}\varphi_p D(X_p).$$
\end{te}

Since every $D(X_p)$ is defined by homogeneous differential equations, it is naturally
graded, the degree of every element being just its degree as a polynomial. The Hilbert series of $D(X_p)$ is known to be $T_{X_p}(1,y)$. In other words, the coefficients of this polynomial are the dimensions of the graded parts (see
\cite{BDR} or \cite[Theorem 11.8]{li}). Then, by the
theorem above, the space $DM(X)$ is also graded, and by Lemma \ref{aes}
we have:

\begin{te}\label{dDM}
$M_X(1,y)$ is the Hilbert series of $DM(X)$.
\end{te}

By comparing this theorem with Proposition \ref{voZ} we recover the following known result, which can be found for example in (\cite[Chapter 13]{li}) :
\begin{co}
The dimension of $DM(X)$ is equal to the volume of the zonotope $\mathcal{Z}(X)$.
\end{co}

\section{The case of root systems}
This section is devoted to describing a remarkable class of examples. We will assume standard notions about root
systems, Lie algebras and algebraic groups, which can be found
for example in \cite{HA} and \cite{Hu}.

Let $\Phi$ be a root system, $\langle\Phi^\vee\rangle$ be the
lattice spanned by the coroots, and $\Lambda$ be its dual
lattice (which is called the \emph{cocharacter} lattice).
Then we define, as in Section 4.2, a torus $T=T_\Lambda$ having $\Lambda$
as its group of characters. In other words, if $\mathfrak{g}$
is the semisimple complex Lie algebra associated to $\Phi$ and
$\mathfrak{h}$ is a Cartan subalgebra, $T$ is defined as the
quotient $T\doteq {\mathfrak{h}}/{\langle\Phi^\vee\rangle}$.

Each root $\alpha$ takes integer values on
$\langle\Phi^\vee\rangle$, so it induces a character
$$e^{\alpha}: T\rightarrow {\mathbb{C}}/{\mathbb{Z}}\simeq \mathbb{C^*}.$$
Let $X$ be the set of these
characters. More precisely, since $\alpha$ and $-\alpha$ define
the same hypersurface, we set
$$X\doteq \left\{e^{\alpha},\:\alpha\in \Phi^+\right\}.$$
In this way a toric
arrangement is associated to every root system.

\begin{re}

~

\begin{enumerate}
  \item Let $G$ be the semisimple, simply connected linear
      algebraic group associated to $\mathfrak{g}$. Then
      $T$ is the maximal torus of $G$ corresponding to
      $\mathfrak{h}$, and $\rx$ is known as the set of
      \emph{regular points} of $T$.
  \item One may take as $\Lambda$ the root lattice (or equivalently, replace the coroot lattice by the character lattice in the construction above). But then one obtains as $T$ a maximal
      torus of the semisimple \emph{adjoint} group $G^a$,
      which is the quotient of $G$ by its center.
\end{enumerate}
\end{re}

Toric arrangements defined by root systems have been studied in \cite{Mo}; we now show two applications to the present work.
Let $W$ be the (finite) Weyl group of $\Phi$, and let $\widetilde{W}$ be the associated affine Weyl group. We denote by $s_0,\ldots, s_n$ its generators, and by $W_k$ the subgroup of $\widetilde{W}$ generated by all the elements $s_i$ but $s_k$.  Let $\Phi_k\subset \Phi$ be the root system of $W_k$, and denote by $X_k$ the corresponding sublist of $X$. Then we have:
\begin{co}
$$M_X(1,y)=\sum_{k=0}^n \frac{|W|}{|W_k|} T_{X_k}(1,y).$$
\end{co}
\begin{proof}
Straightforward from \cite[Theor. 1]{Mo} and Lemma \ref{aes}.
\end{proof}

Furthermore, in \cite{Mo} the following theorem is proved. Let $W$ be the Weyl
group of $\Phi$.
\begin{te}
The Euler characteristic of $\rx$ is equal to $(-1)^n |W|$.
\end{te}
By comparing this statement with Corollary \ref{MEu}, we get the
following
\begin{co}\label{wey}
$$M_X(1,0)= |W|.$$
\end{co}
It would be interesting to have a more direct proof of this
fact.

\begin{ex}
The toric arrangement described in Example \ref{exa} correspond to the root system of type $C_2$. Notice that the order of the Weyl group of type $C_2$ is
$$8=P(-1)=M_X(1,0)=R(X).$$
\end{ex}

~

~


\end{document}